\documentclass[a4paper,reqno,11pt]{amsart}

\usepackage[all]{xy}
\usepackage[english]{babel}
\usepackage{amssymb}
\usepackage{graphicx}
\usepackage[legalpaper, margin=1in]{geometry}
\usepackage{bigstrut}
\usepackage{amsmath,amsthm,amssymb,xypic,verbatim}

\newcommand{\p}{\mathbb{P}} 

\newcommand{\F}{\mathcal{F}}
\newcommand{\N}{\mathbb{N}}

\newcommand{\C}{\mathbb{C}}

\newcommand{\LD}{\mathbb{L}}
\newcommand{\vu}{\varnothing}
\newcommand{\I}{\mathcal{I}}
\newcommand{\E}{\mathcal{E}}

\newcommand{\X}{\mathcal{X}}

\newcommand{\A}{\mathbb{A}}

\newcommand{\LL}{\mathcal{L}}
\newcommand{\M}{\mathcal{M}}
\newcommand{\oo}{\mathcal{O}}
\newcommand{\mul}{\mathop{\rm mult}\nolimits}
\newcommand{\Bs}{\mathop{\rm Bs}\nolimits}

\newcommand{\Bl}{\mathop{\rm Bl}\nolimits}

\newcommand{\Span}[1]{\langle#1\rangle}
\newcommand{\expdim}{\mathop{\rm edim}\nolimits}
\newcommand{\vdim}{\mathop{\rm vdim}\nolimits}
\newcommand{\h}{\mathop{\rm h}\nolimits}

\theoremstyle{plain}

\newtheorem{thm}{Theorem}
\newtheorem{pro}[thm]{Proposition}
\newtheorem{lem}[thm]{Lemma}   
\newtheorem{cor}[thm]{Corollary}

\theoremstyle{definition}
\newtheorem{construction}[thm]{Construction}

\newtheorem{obs}[thm]{Observation}
\newtheorem{es}[thm]{Example}
\newtheorem{con}[thm]{Conjecture}

\newtheorem{dfn}[thm]{Definition}

\newtheorem{rmk}[thm]{Remark}

\newcommand{\mult}[0]{\operatorname{mult}}

\begin{document}

\title[Collisions of fat points and interpolation theory]{Collisions of fat points and\\ applications to interpolation theory}

\author{Francesco Galuppi}
\address{Max Planck Institute for Mathematics in the Sciences\\
Inselstra\ss{}e 22\\
04103 Leipzig\\
Germany}
\email{galuppi@mis.mpg.de}
\thanks{}
\begin{abstract}
  We address the problem to determine the flat limit of the collision of fat points in $\p^n$. We give a description of the limit scheme in many cases, in particular in low dimension and multiplicities.  
  The problem turns out to be closely related with interpolation theory, and as an application we exploit collisions to prove new cases of Laface-Ugaglia Conjecture.
\end{abstract}
\maketitle	

The study of linear systems on projective varieties is an important branch of algebraic geometry, and it is studied since the beginning of 20th century. Despite the efforts of many mathematicians, there is still much we do not know.
Interpolation theory deals with linear systems of divisors passing through a bunch of fixed points with prescribed multiplicities. Given such a linear system, it is natural to ask its dimension, and sometimes the na\"{\i}f parameter counting does not give the correct answer. A challenging problem in interpolation theory is the study of special linear systems, that is, systems having larger dimension than the expected one. In general, this is widely open, even for systems of plane curves.
A standard approach is to consider degenerations. By semicontinuity, a degeneration cannot decrease the dimension of a system, so, if the degenerated system is non-special, then the original one is non-special as well. Typically, it is convenient to degenerate some of the assigned base points to a special configuration, for instance by sending them on a hyperplane in order to apply induction arguments. Sometimes it can be useful to allow points not only to be in special position, but also to collide to the same point. This idea was introduced by Evain in \cite{EvainIdea}, in order to study the dimension of systems of plane curves.

The degenerated linear system features a new singular base point, which is the limit of the collision. Hence this degeneration strategy is useful only if we fully understand the limit scheme. This raises a fairly natural question, which is of interest in itself.
\vskip4pt
	{\bf Question.} Given $n,m_1,\dots,m_h\in\N$, what is the flat limit of $h$ colliding fat points of multiplicities $m_1,\dots,m_h$ in $\p^n$? Or more generally on a smooth $n$-dimensional variety?
\vskip4pt
While it is easy to ask, this question has not a simple answer. Results in \cite{CM3} and in \cite{Ne} show the lack of a clean and complete solution even if $n = 2$.
Since the beginning of this work, we realized that there is a nice interplay between collisions of fat points and interpolation theory. On one hand, some basic properties of the limit of a collision can be stated in the language of linear systems. On the other hand, with this technique we can afford new ways to degenerate a bunch of singular points, so we have new tools to compute the dimension of linear systems. We believe this connection to be worth of a deep analysis.

Our strategy to describe a collision is quite simple. Thanks to flatness, we know the degree of the limit scheme. First, we compute the multiplicity of the limit. In Proposition \ref{pro:molteplicità_limite} we prove that this is actually a question about linear systems. Once the multiplicity is determined, we try to get more information about this linear system, such as its base locus, in order to get further conditions on the limit. In this way we get a candidate scheme, and we compute its degree. Since this candidate is a subscheme of the limit, if their degree coincide then they are the same scheme.

Once we are able to describe a limit, we can use it to degenerate linear systems. In Section \ref{section:interpolation} we use such degenerations to study some families of linear systems on $\p^2$, $\p^3$, $\p^4$ and $\p^1\times\p^1$, and we compare our results with the known results in interpolation theory. For instance, the Laface-Ugaglia conjecture predicts that the linear system $ L_{3,d}(m^r)$ is non-special for $d>\frac{3}{\sqrt{2}}m+o(m)$. We will prove the following partial result in this direction.
\begin{thm}\label{thm:LUcondgrande}
	If $r\le 15$ and $d\ge 3m$, then $ L_{3,d}(m^r)$ is non-special. In particular, Laface-Ugaglia conjecture holds for these values $d$, $m$ and $r$.
\end{thm}
Sometimes it is convenient to work out examples with a software. In those cases we use Macaulay2, available at {\tt{www.math.uiuc.edu/Macaulay2}}.

\section{Notations and preliminaries}\label{sec:notations}
The main aim of this section is to provide the setup for the rest of the paper. After establishing the definitions we need about points with multiplicities and linear systems, we recall two important techniques to bound the dimension of a given system, namely restriction and specialization. Classically, the latter roughly consist of moving some of the imposed points to a special position. A variation of the standard specialization is presented in Construction \ref{cons:collapse}, where the points are allowed not only to be moved to a special position, but also to collapse. Proposition \ref{pro:molteplicità_limite} computes the multiplicity of the resulting limit scheme and points out the connection with interpolation theory. Lemma \ref{lem:grassisugrasso} is the prototype of a collision of fat points, and will be useful in the rest of the paper.

We work over the complex field $\C$. Every scheme will be projective, unless we specify it is not. For a scheme $X$ and a closed subscheme $Y\subset X$, we will write $\I_{Y,X}$ to denote the ideal sheaf of $Y$ in $X$. 
If no ambiguity is likely to arise, we will write simply $\I_Y$ instead of $\I_{Y,X}$. If $\F$ is a coherent sheaf on $X$ and $i\in\N$, we will write $H^i\F$ to denote the cohomology group $H^i(X,\F)$ and $\h^i\F$ for its dimension.

\begin{dfn} Let $X$ be a $0$-dimensional scheme. The \textit{degree}, or \textit{length}, of $X$, denoted by $\deg X$, is the dimension of its ring of regular functions as a complex vector space. If $X$ is supported on a point $p$, we define the \textit{multiplicity} of $X$, denoted by $\mul X$, to be the largest $k\in\N$ such that $X$ contains the $k$-tuple point supported on $p$.
\end{dfn}

There is a more general definition of multiplicity. If $X$ is a scheme of any dimension and $Y$ is an irreducible component of $X_{\mbox{\Small{red}}}$, one can define the multiplicity of $X$ at $Y$ as the length of the local ring $\oo_{X,Y}$ (see \cite[Section 1.2.1]{EisenbudHarris}).

If $X\subset\p^n$ is a 0-dimensional subscheme, then $\deg X$ is the limit value of the Hilbert function of $X$. In other words, if $d$ is large enough, then $X$ imposes $\deg X$ independent linear conditions to degree $d$ divisors of $\p^n$. 
Let us recall a basic fact about 0-dimensional schemes.

\begin{lem}\label{stessogrado} \label{basic} Let $X$ be a $0$-dimensional schemes supported at a point, and let $Y$ be a subscheme of $X$. If $\deg Y=\deg X$, then $Y = X$.	
\end{lem}

Since we will deal with linear systems with assigned singularities
, we introduce the notations we are going to use.
\begin{dfn} Let $V$ be a smooth quasi-projective variety, let $p_1,\ldots, p_r\in V$.
	The \textit{linear system} 
	$$ L_{V,d}(m_1,\dots,m_r)(p_1
	,\ldots,p_r)\subset H^0\oo_{V}(d)$$ 
	is the vector space of divisors of $V$ having multiplicities at least $m_i$ at the point $p_i$.
	In other words, if $X=m_1p_1+\ldots+m_rp_r\subset V$ is a fat point subscheme, then $$L_{V,d}(m_1,\dots,m_r)(p_1
	,\ldots,p_r)=H^0\I_{X,V}(d).$$
	We will write $\LL_{V,d}(m_1,\dots,m_r)(p_1
	,\ldots,p_r)$ to denote the associated ideal sheaf, that is,
	$$\LL_{V,d}(m_1,\dots,m_r)(p_1
	,\ldots,p_r):=\I_{X,V}(d).$$
	If either the points $p_1,\ldots, p_r$
	 are in general position, or no confusion is likely to arise, then we set
	$$ L_{V,d}(m_1
	,\dots,m_r):= L_{V,d}(m_1,\dots,m_r)(p_1
	,\ldots,p_r).
	$$
	Moreover, if $m_1=\ldots=m_s=m$ then we indicate
	$$ L_{V,d}(m^s,m_{s+1},\dots,m_r):= L_{V,d}(m_1,\dots,m_r).$$
	We will write $ L_{n,d}(m_1,\dots,m_r)$ instead of $ L_{\p^n,d}(m_1,\dots,m_r)$. Finally, we will use $ L_{\p^1\times\p^1,(a,b)}(m_1,\dots,m_r)$
to indicate the system of bidegree $(a,b)$ curves on $\p^1\times\p^1$ with the prescribed singularities.
\end{dfn}

Some authors in the literature consider a linear system as a projective space, so they work with $\p(L_{n,d}(m_1,\dots,m_r))$. The two approaches are equivalent. We only have to be aware that in this paper we work with affine dimensions, rather than projective dimensions.

\begin{dfn}
	Let $n:=\dim V$. The \textit{virtual dimension} of such a linear system is 
	$$\vdim L_{V,d}(m_1
	,\dots,m_r):=\h^0\oo_{V}(d)-\sum_{i=1}^r\binom{m_i-1+n}{n}
	.$$
	The \textit{expected dimension} is defined	as
	$$\expdim L_{V,d}(m_1,\dots,m_r):=\max
	\left\{\vdim L_{V,d}(m_1,\dots,m_r),0\right\},$$
	where expected dimension $0$ indicates that the linear system is expected to be empty.
	Note that 
	$$\dim L_{V,d}(m_1,\dots,m_r)\ge\expdim
	 L_{V,d}(m_1,\dots,m_r).$$
	When the linear conditions imposed by the base points are dependent, then previous inequality is strict, and the linear system is said to be \textit{special}. On the other hand, if the conditions are independent, then $\dim L_{V,d}(m_1,\dots,m_r)=\expdim L_{V,d}(m_1,\dots,m_r)$ and the system is called \textit{non-special}.
\end{dfn}

Not much is known about the classification of special linear systems $ L_{n,d}(m_1,\dots,m_r)$ for an arbitrary $n$. The most important result in this direction is the celebrated Alexander-Hirschowitz theorem, proven in \cite{AH}, that solves the problem for systems with general double points.

\begin{thm}[Alexander-Hirschowitz] \label{thm:AH} The linear system $ L_{n,d}(2^h)$ is special if and only if $(n,d,h)$ is one of the following:
	\begin{itemize}
		\item[i)] $(n,2,h)$ with $2\leq h\leq n$,
		\item[ii)] $(2,4,5)$, $(3,4,9)$, $(4,3,7)$, $(4,4,14)$.
	\end{itemize}
\end{thm}

Let us introduce two classical tools to deal with the computation of the dimension of a linear system.
The first one is an useful exact sequence that will help us later.

\begin{dfn}
	\label{dfn:castelnuovo}  
	Let $S\subset V$ be a smooth hypersurface and $L$ a linear system on $V$. Let $\rho: L\to L_{|S}$ be the restriction map. Let $ L-S:=\ker(\rho)$, that is $ L-S=\{0\}\cup\{D\in L\mid D\supset S\}$. Denote by $\LL-S$ the associated sheaf and by $\LL_{|S}$ the sheaf associated to $L_{|S}$. There is a short exact sequence of sheaves on $V$
	\[0\to  \LL-S\to \LL\to \LL_{|S}\to 0,\]
	 called \textit{restriction sequence} or
	\textit{Castelnuovo sequence}.
\end{dfn}

By Castelnuovo sequence, if both $ L-S$ and $ L_{|S}$ are non-special of non-negative virtual dimension, then $ L$ is non-special.

Another thing we can do with a linear system $ L:= L_{V,d}(m_1,\dots,m_r)$ is to degenerate it, namely we can pick $q_1,\dots,q_r\in V$ and move the singularities of $ L$ from general position to the point we choose. In this way we have to deal with
\[ L_0:= L_{V,d}(m_1,\dots,m_r)(q_1,\dots,q_r)\]
instead of $ L$. If we choose the points $q_i$ wisely, hopefully we can say something on $ L_0$ (for instance, on its dimension) and use semicontinuity to get information about $ L$. Now we want to make this intuitive notion more precise. The next definitions are based on \cite{CM1}.

\begin{dfn} Let $Y$ be a smooth variety. A \textit{degeneration} is a proper and flat morphism $\pi: Y\to\Delta$, where $\Delta\ni 0,1$ is a complex disk. For any $t\in\Delta$, we denote by $Y_t$ the fiber of $\pi$ over $t$. Let $\sigma_i:\Delta\to Y$ be sections of $\pi$ and let $Z$ be a scheme supported on $\bigcup_i\sigma_i(\Delta)$. For $t\in\Delta$, define $Z_t:=Z_{|Y_t}$, so that $Z_0$ is the flat limit of the schemes $Z_t$.
We say that $Z_0$ is a \textit{specialization} of $Z_t$.
\end{dfn}

For the sake of simplicity, sometimes we will say that $Z_0$ is a specialization of $Z_1$, instead of $Z_t$, implying that 1 is any general point of $\Delta$.

\begin{construction}[Specialization without collisions]\label{cons:notationlinearsystem}
	Let $m_1,\dots,m_r\in\N$ and let $V$ be a smooth variety. Let $Y:=V\times\Delta$ and let $\pi:Y\to\Delta$ be the projection. Fix $r$ disjoint sections $\sigma_1,\ldots,\sigma_{r}:\Delta\to Y$. 
	Let
	\[Z:=\bigcup_{i=1}^r \sigma_i(\Delta)^{m_i}\subset Y\]
	be the scheme supported on the sections with multiplicity $m_i$ along $\sigma_i(\Delta)$.
	Let 
	$$\LD 
	:=H^0\I_{Z,Y}(d)$$
	be the linear system on $Y$ associated to degree $d$ divisors having multiplicities at least $m_i$ along $\sigma_i(\Delta)$. 
	Then, for a general $t\in \Delta$, the linear system
	$\LD 
	_{|Y_t}$ coincides with 
	$$ L_t:= L_{V,d}(m_1,\dots,m_r)(\sigma_1(t),\ldots,\sigma_r(t)).$$
	\label{rem:specialize_to_get_non_speciality}
	By semicontinuity, we have	$$\dim L_0\geq \dim L_t.$$
	Therefore, in order to prove that $ L_t$ is non-special, it is enough to produce a degeneration such that 	$ L_0$ is non-special.
\end{construction}

In this paper we are interested in a different kind of degeneration, namely we want to drop the hypothesis that $\sigma_1,\dots,\sigma_r$ are disjoint. We now modify Construction \ref{cons:notationlinearsystem} in order to allow the specialized points to collapse. Since a limit is a local technique, we can work with the affine space instead of a variety $V$. This idea is based on \cite{CM3}.

\begin{construction}[Specialization with $h$ collapsing points]\label{cons:collapse}
	Let $Y:=\A^n\times\Delta$, with second projection 
	$\pi:Y\to\Delta$ and fibers $Y_t:=\A^n\times\{t\}$. 
	Fix a point $q\in Y_0$ and $h$ general sections $\sigma_1,\ldots,\sigma_{h}:\Delta\to Y$ of $\pi$ such that $\sigma_i(0)=q$.
	Define $$Z:=\bigcup_i\sigma_i(\Delta)^{m_i}\subset Y.$$	
	Let $X\to Y$ be the blow-up of $Y$ at the point $q$, with exceptional divisor $W$. Then we have 
	a degeneration $\pi_{X}:X\to\Delta$ and sections $\sigma_{X,i}:\Delta\to X$. The fiber $X_0$ is reducible, and it is given by $W\cup \tilde{Y_0}$, where $W\cong\p^n$ and $\tilde{Y_0}=\Bl_q Y_0$ is $\A^n$ blown up at one point. Let $R=W\cap \tilde{Y_0}\cong\p^{n-1}$ be the exceptional divisor of this blow-up. We want to stress that, since the sections $\sigma_i$ are general, $\sigma_{X,1}(0),\dots,\sigma_{X,h}(0)$ are general points of $W$. With these notations, we say that $Z_0=Z_{|Y_0}$ is the flat limit of $h$ collapsing points of multiplicities $m_1,\dots,m_h$.
\end{construction}

 Our goal will be to describe $Z_0$. Once we understand the limit, we may study the dimension of a linear system via its specializations with collapsing points, using the same technique described in Construction~\ref{rem:specialize_to_get_non_speciality}.

\begin{rmk} \label{rmk:possiamo assumere di essere nell'affine}\begin{enumerate}
		\item Since a collision is a local construction, our results about collisions on $\A^n$ hold on any smooth variety.
		\item When we consider degenerations as in Construction \ref{cons:notationlinearsystem} or \ref{cons:collapse}, by flatness we know that the length is preserved, so $\deg Z_0=\deg Z_1$.
		\item One could give the same definitions without requiring that the sections $\sigma_i$ are general, but in this case the theory becomes more involved and less interesting for applications.
	\end{enumerate}
\end{rmk}

As a warm-up, we start with an easy result that describes collisions of fat points on smooth curves.

\begin{pro}\label{pro:sullaretta}
	Let $m_1,\dots,m_h\in\N$ and let $m=m_1+\ldots+m_h$. The limit of $h$ collapsing points of multiplicities $m_1,\dots,m_h$ in $\A^1$ is an $m$-tuple point.
	\begin{proof}
		It is enough to observe that the only length $m$ subscheme of $\A^1$ supported at a point is the $m$-tuple point.
	\end{proof}
\end{pro}	

Now that the case $n=1$ is settled, for the rest of this paper we assume $n\ge 2$ and we try to move to some more interesting cases in higher dimension. In order to understand what $Z_0$ is, the first problem to tackle is to compute its multiplicity. 
In \cite[Proposition 4]{EvainIdea}, the author solves the problem for $n=2$, under the assumption that the sections $\sigma_i$ are given by a homothetic transformation. It is proved that the multiplicity of the limit scheme $Z_0$ is the minimum integer $j$ such that the linear system $H^0\I_{Z_1,\A^2}(j)=L_{n,j}(m_1,\dots,m_h)$ is nonzero. In \cite[Theorem 2.6]{Ne}, the assumption on the section is dropped, but only a bound is provided. In \cite{GM}, a different proof allows the authors to generalize this bound to any dimension. Now we want to improve this result, and show that the estimated value is actually achieved with equality.
\begin{pro}\label{pro:molteplicità_limite}
	Let $k:=\min\{j\in\N\mid H^0\I_{Z_1,\A^n
	}(j)\neq 0\}$. Then $\mul Z_0=k$. In particular, the multiplicity of the limit scheme does not depend on $\sigma_i$, as long as they are general.
	\begin{proof}
		Thanks to \cite[Lemma 20]{GM}, it suffices to prove that $\mul Z_0\le k$, so we only have to show that $Z_0$ is contained in a degree $k$ divisor. For $t\neq 0$, set $l=\h^0\I_{Z_t}(k)$. Since the points of $Z_t$ are in general position, $l$ does not depend on $t$, and by hypothesis $l\ge 1$. Let $P\subset Y_t=\A^n$ be a set of $l-1$ general points, and define $Z'_t=Z_t\cup P$.
		Observe that $Z'_t\supset Z_t$ for every $t$, and there is a unique degree $k$ divisor $D_t\subset Y_t$ such that $D_t\supset Z'_t$. Let $f_t$ be the polynomial defining $D_t$ as a divisor in $Y_t$.
		Then $f_0$ defines a divisor $D_0$ of $Y_0$ which is the flat limit of the $D_t$'s. Hence $\deg f_0\le \deg f_t=k$ 
		 and $D_0\supset Z'_0\supset Z_0$, so $\mul Z_0\le k$.
	\end{proof}
\end{pro}

In some cases the multiplicity is enough to compute the limit scheme.



\begin{lem}\label{lem:grassisugrasso} Let $m_1,\dots,m_h,m,n\in\N$ be such that
\[\binom{m_1+n-1}{n}+\ldots+\binom{m_h+n-1}{n}=\binom{m+n}{n}.\]
If 
$ L_{n,m}(m_1,\dots,m_h)$ is non-special, then the limit of $h$ colliding points of multiplicities $m_1,\dots,m_h$ in $\A^n$ is a $(m+1)$-tuple point.
			\begin{proof} Consider the scheme $Z_1\subset\A^n$ made by $h$ general fat points of multiplicities $m_1,\dots,m_h$. By hypothesis, $L_{n,m}(m_1,\dots,m_h)=H^0\I_{Z_1,\A^n}(m)$ is non-special and it has expected dimension $0$, so it is empty. On the other hand, $ L_{n,m+1}(m_1,\dots,m_h)$ has positive expected dimension, so it is not empty. By Proposition \ref{pro:molteplicità_limite}, the limit scheme $Z_0$ contains a $(m+1)$-tuple point. We know that $\deg Z_0=\deg Z_1$ by flatness, and by hypothesis the length of $Z_1$ coincides with that of a $(m+1)$-tuple point. We conclude by Lemma \ref{stessogrado}.
		\end{proof}\end{lem}

The previous Lemma will be very useful for our purposes. Indeed, when we use limits to specialize a linear system, the most effective result would be a description of $Z_0$ as a fat point of some multiplicity. However, this can happen only if the hypothesis of Lemma \ref{lem:grassisugrasso} are satisfied. When the scheme we are specializing does not have the degree of a multiple point, this analysis is not enough to determine the limit scheme.

In the notations of Construction \ref{cons:collapse}, 
%
let \[\Sigma:=\bigcup_{i=1}^h\sigma_{X,i}(\Delta)\subset X\] be the smooth scheme associated to strict transform $Z_X$ of $Z$ on $X$.
Let $\X\to X$ be the blow-up of the ideal sheaf $\I_{\Sigma}$ with exceptional divisors $\E_1,\dots,\E_h$, and let $\varphi:\X\to\Delta$ be the degeneration onto $\Delta$. Note that this blow-up is an isomorphism in a neighbourhood of $Y_0$. The central fiber is
$$\X_0:=\varphi^{-1}(0)=P\cup \tilde{Y}_0,$$
where $P$ is the blow-up of $W\cong\p^n$ at $h$ general points. With abuse of notation, we identify $R\subset X$ with its strict transform $R=P\cap \tilde{Y}_0$. The linear systems we are interested in are $ L:=H^0\oo_\X(-\sum_i m_i\E_i-\mul(Z_0)P)$ and its restrictions $ L_P$, $ L_R$, to $P$ and $R$.
The linear system $ L$ is complete. However, the following example (\cite[Example 2.10]{Ne}) shows that in general $ L_R$ is not complete.
\begin{es}\label{3su3plo}
	In the case of 3 colliding double points in $\A^2$, the limit has multiplicity 3. On the other hand, a triple point has degree $6$, while $\deg Z_1=9$, therefore the limit is not only the triple point. 
	 In order to better understand the first infinitesimal neighborhood of the limit, we look at the system $ L_P\cong L_{2,3}(2^3)$. A plane cubic with 3 double points in general position is the union of 3 lines, that intersect $R$ in 3 points. Hence $ L_R\subsetneq\oo_R(3)$ is not complete, but rather it has only one nonzero section, consisting of the 3 intersection points. Those 3 intersections are base points for $ L_P$, so they are tangent directions (infinitely near points) in the limit scheme. We proved that $Z_0$ contains a triple point with 3 infinitely near simple points. Since these tangent directions impose independent conditions on cubics of $R$, that is, $ L_R$ is non-special, this subscheme of $Z_0$ has degree $6+3=9=\deg Z_0$. By Lemma \ref{stessogrado}, $Z_0$ is a triple point with 3 infinitely near simple points.
\end{es}

It is worth to mention that, unlike $ L_R$, the system $ L_P$ is always complete, as proven in \cite[Lemma 24]{GM}. Observe that in Example \ref{3su3plo} we computed the degree of our candidate by checking that $ L_R$ is non-special. This is an important and often nontrivial step, as we will see in Section \ref{sec:doublepoints}. In order to make this precise, we need a lemma.


\begin{lem}
	\label{lem:degree of a fat point with infinitely near simple points}
	Let $q\in\A^n$. Let $E$ be the exceptional divisor of $\Bl_q\A^n$ and $B:=\{p_1,\dots,p_t\}\subset E$ a set of $t$ simple points. Let $X$ be the scheme supported at $q$ consisting of a $m$-ple point with $t$ infinitely near points $p_1,\dots,p_t$. Then
	\[\deg X=\binom{n+m-1}{n}+\binom{n+m-1}{n-1}-\h^0\I_{B,E}(m)=\binom{n+m}{n}-\h^0\I_{B,E}(m).\]
	\begin{proof}
		We argue by induction on $t$. If $t=0$, then $X$ is just a $m$-ple point in $\A^n$, so $\deg X=\binom{n+m-1}{n}$. On the other hand, $B=\varnothing$, so $\h^0\I_{B,E}(m)=\h^0\oo_{\p^n}(m)=\binom{n+m-1}{n-1}$ and the statement holds. Assume then $t\ge 1$. Let $B':=B\setminus\{p_t\}$ and let $X'\subset X$ be the subscheme consisting of a $m$-ple point with $t-1$ infinitely near points $p_1,\dots,p_{t-1}$. By induction hypothesis,
			\[\deg X'=\binom{n+m}{n}-\h^0\I_{B',E}(m).\]
			There are two possibilities. If $p_t$ is a base point for the linear system $H^0\I_{B',E}(m)$, then 
			 $\I_{B',E}(m)=\I_{B,E}(m)$, so
			\[\deg X=\deg X'=\binom{n+m}{n}-\h^0\I_{B,E}(m).\]
			If $p_t$ is not a base point for $H^0\I_{B',E}(m)$, then $\h^0\I_{B',E}(m)=1+\h^0\I_{B,E}(m)$ and $X'$ is a proper subscheme of $X$, so $\deg X'<\deg X$. Since $p_t$ is a simple point, the difference of the degrees cannot be more than 1, so
			\begin{align*}
			\deg X&=1+\deg X'=1+\binom{n+m}{n}-\h^0\I_{B',E}(m)\\
			&=1+\binom{n+m}{n}-(1+\h^0\I_{B,E}(m)) 
			.\qedhere
			\end{align*}

	\end{proof}
	\end{lem}

Notice that Lemma \ref{lem:degree of a fat point with infinitely near simple points} does not need to hold when the infinitely near points have multiplicities greater than 1. An explicit computation shows that the subscheme of $\A^3$ consisting of a triple point with an infinitely near double point has degree 12. Nonetheless, we will often encounter infinitely near simple points. In those cases, the lemma allows us to get information on the limit scheme.

\begin{cor}\label{cor:condizioni imposte dai punti infinitamente vicini}
	In the notation of Construction \ref{cons:collapse}, let $m=\mul Z_0$ and let $B:=(\Bs( L_P))_{|R}$. If $B$ is a set of simple points, then $\deg Z_0\ge\binom{n+m}{n}-\dim( L_R)$. In particular, if $ L_R$ is non-special then $\deg Z_0\ge\binom{n+m-1}{n}+t$.
	\begin{proof}
By hypothesis $Z_0$ contains a subscheme $S$ consisting of a $m$-ple point with some infinitely near simple points. We conclude by Lemma \ref{lem:degree of a fat point with infinitely near simple points}.
	\end{proof}
\end{cor}

Before we move to the first results on limits, it is important to have clear in mind what kind of characterization we want. In general it will be too complicated to determine the limit up to isomorphism. For instance, let $Z_1$ consist of 14 general simple points in $\A^2$ and let $Z_0$ be their collision. By Proposition \ref{pro:molteplicità_limite}, $\mul Z_0=4$. However, just as in Example \ref{3su3plo}, the scheme cannot be just a fourtuple point. In order to find more information, we look at the linear system $ L_P\cong L_{2,4}(1^{14})$. This system has only one nonzero section $C$, so its restriction to the exceptional line $R$ consists of 4 simple points. Thus the candidate limit is a fourtuple point with 4 tangent directions, and since $ L_R$ is non-special it has degree at least
\[\binom{3+2}{2}+4=14=\deg Z_1\]
by Corollary \ref{cor:condizioni imposte dai punti infinitamente vicini}. Therefore, our candidate actually coincides with the limit. However, notice that if we change the sections $\sigma_1,\dots,\sigma_{14}$, then we will have different tangent directions to the limit. Recall that two 4-tuples of points in $\p^1$ are not projectively equivalent in general, so the limits do not need to be isomorphic. Nonetheless, we will be satisfied to say that the limit is a fourtuple point with 4 infinitely near simple points. 

We also want to stress that our analysis works as long as we make all points collide at once. If we collide some of them to a limit scheme $\tilde{Z_1}$ and then we collide the others together with $\tilde{Z_1}$, we are not guaranteed to obtain the same limit scheme as if we collide all of them at once.	
As an example, let $Z_1$ be the scheme consisting of a double point and 3 simple points in $\A^2$. If we make them collide, the multiplicity of the limit scheme $Z_0$ is 3 by Proposition \ref{pro:molteplicità_limite}. On the other hand, we could collide the 3 simple points to a double point, but the limit of 2 colliding double points has multiplicity 2. However, this kind of multi-staged collisions provides new legitimate ways to degenerate a linear system.

It is time to move to the description of the limit $Z_0$, and we start with the limit of a bunch of colliding double points. While in some sense it is simpler, the study of collisions of simple points requires a different and peculiar treatment, and will be addressed in another paper.

\section{Double points}\label{sec:doublepoints}
In this section we assume that all the collapsing points have multiplicity 2. First we review some of the well-understood cases. In Definition \ref{dfn:candidate} we explicitly construct a scheme consisting of one triple point with some tangent directions, we prove that it is a subscheme of the limit and we conjecture that they coincide. This is equivalent to compute the number of independent conditions given by a set of simple points in a given special position, see Conjecture \ref{con:n+k}. Lemma \ref{lem:quantecodizionidanno} allows us to prove that the conjecture holds for several small values. Conversely, in Theorem \ref{thm:lift} we prove that the general triple point with the suitable number of tangent direction can be obtained as a collision of double points.

When dealing with linear systems with double points, we will repeatedly use Theorem \ref{thm:AH}. As a warm-up, we deal with the cases satisfying the hypothesis of Lemma \ref{lem:grassisugrasso}.

\begin{pro}\label{pro:doppisumultiplo}
	Let $m\ge 2$ and $(n,m)\notin\{(2,3),(2,5),
	(4,4),(4,5)\}$. Define $h:=\frac{\binom{n+m-1}{n}}{n+1}$. If $h\in\N$, then the limit of $h$ colliding double points in $\p^n$ is an $m$-tuple point.
	\begin{proof}
		By our numerical assumption, Theorem \ref{thm:AH} implies that $ L_{n,m-1}(2^h)$ is non-special. Now the statement follows by Lemma \ref{lem:grassisugrasso}.
	\end{proof}
\end{pro}

As we already noticed, in most cases the limit is not just a point with multiplicity. As Example \ref{3su3plo} shows, once we understand the minimum degree of a divisor containing $Z_1$, we need information on the base locus of such divisors.

When we deal with double points, it is convenient to work in the case $h>n$. Indeed, $h\le n$ yields $\mul Z_0=2$, and $ L_{n,2}(2^h)$ has a nonreduced base locus, so it is difficult to describe the conditions imposed on the limit linear system. On the other hand, if $h>n$ then we have $\mul Z_0=3$, at least for $n$ big enough, and the base locus of cubics with assigned double points is very well understood. We start with a technical result.

\begin{lem}\label{lem:indip} Let $n\geq 2$ and $l\le n+2$. Let $A:=\{a_1,\dots,a_{l}\}$ be a set of $l$ general points in $\p^n$ and let $R$ be a hyperplane such that $A\cap R=\varnothing$. Let $p_{ij}:=\langle a_i,a_j\rangle\cap R$ and
	$$B:=\{p_{ij}\mid 1\le i<j\le l\}.$$
	Then $ L_{n-1,2}(B)$ and $ L_{n-1,3}(B)$ are non-special, that is, the points of $B$ impose independent conditions to quadrics and cubics of $R$.
	\begin{proof}
		It is enough to prove the claim for $l=n+1$ for quadrics and $l=n+2$	for cubics.	$ L_{n-1,2}(B)$ is non-special by \cite[Lemma 25]{GM}.
		
		Now assume that $l=n+2$. We prove that the points of $B$ are general for cubics by induction on $n$. It is easy to check that the statement holds for $n=2$, so we assume $n\ge 3$. Specialize $a_1,\dots,a_{n+1}$ on a general hyperplane $W\subset\p^{n}$. Define
		$$B_1:=\{p_{ij}\mid 1\le i<j\le n+1\}\mbox{ and }B_2:=\{p_{1,n+2},\dots,p_{n+1,n+2}\}.$$
		Observe that the points of $B_2$ are in general position on $R$, and $B=B_1\cup B_2$. Let $H:=W\cap R=\p^{n-2}$. Castelnuovo exact sequence reads
		\[0\to\I_{B_2,R}(2)\to\I_{B,R}(3)\to\I_{B_1,H}(3)\to 0.\]
		Since $B_2$ is a set of general points of $R$, $\h^1\I_{B_2,R}(2)=0$. If we set
		$$A_1:=\{a_1,\dots,a_{n+1}\},$$
		then $A_1$ is a set of general points in $W$ and $H$ is an hyperplane of $W$ such that $A_1\cap H=\vu$. By induction hypothesis, $\h^1\I_{B_1,H}(3)=0$. Hence $\h^1\I_{B,R}(3)=0$ and so $B$ imposes independent conditions on cubics of $R$. Since the conditions are independent when $A$ and $B$ are in this specialized configuration, the statement holds by semicontinuity.
	\end{proof}
\end{lem}

\begin{rmk}\label{rmk:nongenerici}
Even if $B$ imposes independent conditions, the points of $B$ are not in general  position.
	For every choice of $t$ points of $A$, their span is a $\p^{t-1}$, so the corresponding $\binom{t}{2}$ points of $B$ lie on a $\p^{t-2}$.
\end{rmk}

The next two Propositions, proven in \cite{GM}, solve the cases $h=n+1$ and $h=n+2$.

\begin{pro} \label{collision1} If $n\ge 2$, then the limit of $n+1$ collapsing double points in $\A^n$ is a triple point with $\binom{n+1}{2}$ tangent directions. The infinitely near simple points are in the special position described by Remark \ref{rmk:nongenerici}.
\end{pro}

\begin{pro} \label{collision2} Let $Z_0$ be the limit of $n+2$ collapsing double points in $\A^n$.
	\begin{enumerate}
		\item If $n=2$, then $Z_0$ is a $4$-tuple point, together with the involution described in \cite[Proposition 3.1]{CM3}.
		\item If $n=3$, then $Z_0$ is a $4$-tuple point.
		\item If $n\ge 4$, then $Z_0$ is a triple point with $\binom{n+2}{2}$ tangent directions. In this case the infinitely near simple points are in the special position described by Remark \ref{rmk:nongenerici}.
	\end{enumerate}
\end{pro}

The proofs rely on Proposition \ref{pro:molteplicità_limite} to compute the multiplicity. Then we determine the base locus of the linear system and we apply Lemma \ref{lem:indip} to check that $ L_R$ is non-special, so that we can conclude by Corollary \ref{cor:condizioni imposte dai punti infinitamente vicini} and Lemma \ref{stessogrado}.

Despite the previous results, the limit scheme can be more complicated than a fat point with a bunch of infinitely near points. 
For instance, the limit of 5 colliding double points in the plane is described in \cite[Proposition 3.1]{CM3} as a fourtuple point with a pair of infinitely near tacnodal points.
We could try to apply the argument of Propositions \ref{collision1} and \ref{collision2} to an higher number of colliding double points. Anyway, we cannot expect the same proof to work, because Lemma \ref{lem:indip} does not hold for $l\ge n+3$. As an example, let us work out one of the exceptions of Theorem \ref{thm:AH}.

\begin{es}\label{es:7doppiinp4}
	Consider a set of general points $A:=\{a_1,\dots,a_7\}\subset\p^4$. As in Lemma \ref{lem:indip}, let $R$ be a hyperplane such that $A\cap R=\varnothing$ and $p_{ij}:=\langle a_i,a_j\rangle\cap R$. Then $B:=\{p_{ij}\mid 1\le i<j\le 7\}$ has 21 points, while $\h^0\oo_R(3)=20$, so $L_{R,3}(B)=H^0\I_{B,R}(3)$ is expected to be empty. However, we know that there is a cubic $C\subset\p^4$ singular at $a_1,\dots,a_7$. $C$ contains all the lines joining pairs of points of $A$, so in particular $C_{|R}\supset B$. Consider Castelnuovo exact sequence
	\[0\to \LL_{4,2}(2^7)\to\LL_{4,3}(2^7)\to\I_{B,R}(3)\to 0.
	\]
	Observe that $L_{4,2}(2^7)=0$, so the restriction $ L_{4,3}(2^7)\to H^0\I_{B,R}(3)$ is injective and therefore 
	$C_{|R}$ is a nonzero element of $H^0\I_{B,R}(3)$. Since $\h^0\oo_{\p^3}(3)=20$, the 21 points of $B$ impose at most 19 independent conditions on cubics of $R$. A software computation shows that $B$ actually imposes exactly 19 independent conditions.
\end{es}

More generally, let $Z_1$ be a scheme of $n+3$ double points, with $n\ge 5$. Observe that $\deg Z_1=(n+1)(n+3)$ and $\mul Z_0=3$. It is easy to see that $\Bs L_{n,3}(2^{n+3})$ consists of the double points and of the $\binom{n+3}{2}$ lines joining the pair of points. Then we have $\binom{n+3}{2}$ simple points infinitely near to the limit triple point. However, these simple points do not impose independent conditions to cubics. Indeed, if they did, then Corollary \ref{cor:condizioni imposte dai punti infinitamente vicini} would imply
\[\deg Z_0\ge\binom{n+2}{2}+\binom{n+3}{2}=n^2+4n+4=1+\deg Z_1.\]
Hence those $\binom{n+3}{2}$ simple points impose dependent conditions on $Z_0$. On the other hand, at least $\binom{n+2}{2}$ of them are independent by Lemma \ref{lem:indip}. How can we give a description of the limit is these cases?

\begin{rmk}\label{rmk:graziekris}
	Let $Z$ be an $m$-tuple point supported at $q\in\p^n$, with an infinitely near simple point, and let $l$ be the line containing $q$ corresponding to the infinitely near point. The restriction of $Z$ to a general line through $q$ is an $m$-tuple point, while $Z_{|l}$ has multiplicity $m+1$. This suggests a possible description of the limit of $n+k$ collapsing double points. Assume that $\mul Z_0=3$, and let $l_1,\dots,l_{\binom{n+k}{2}}$ be the base lines, all passing through the limit point $q$. Let $S^4_i$ be the multiplicity 4 subscheme of $l_i$ supported at $q$. We know that $Z_0$ contains the union of the $S^4_i$'s, and we conjecture that they coincide. Now we want to precisely formulate the problem and to provide a solution for small $k$.
\end{rmk}

\begin{dfn}\label{dfn:candidate}
	Let $n,m\ge 2$, and let $l_1,\dots,l_t\subset\A^n$ be lines meeting at the origin. Let $S^m_i$ be the 0-dimensional degree $m$ subscheme of $l_i$ supported at the origin, and let $ I_{S^m_i,\A^n}$ be the ideal defining $S^m_i$ in $\A^n$. Define $Z_n(l_1,\dots,l_t)$ to be the union scheme associated to the ideal $$ I_n(l_1,\dots,l_t):= I_{S^m_1,\A^n}\cap\ldots\cap I_{S^m_t,\A^n}.$$
	If $l_1,\dots,l_t$ are general lines through the origin and $m=4$, then we define $$Z_{n,t}:=Z_n(l_1,\dots,l_t)\mbox{ and }\  I_{n,t}:= I_n(l_1,\dots,l_t).$$
	When $\mul Z_{n,t}=3$, we can think of this scheme as a triple point with $t$ infinitely near simple points, representing the directions corresponding to $l_1,\dots,l_t$.
	
\end{dfn}

\begin{rmk}\label{rmk:limit_description}
	Consider $n+k$ colliding double points in $\A^n$ and assume the limit has multiplicity 3. Then the limit triple point has $\binom{n+k}{2}$ infinitely near simple points, in special position, giving possibly dependent conditions on cubics. Nevertheless, the restriction of the limit scheme to one of the $\binom{n+k}{2}$ corresponding lines $l_1,\dots,l_{\binom{n+k}{2}}$ has degree strictly greater than 3. In particular the limit scheme contains $Z_n\left( l_1,\dots,l_{\binom{n+k}{2}}\right) $. So if we prove that they have the same degree, then we get an explicit description of the limit scheme.
\end{rmk}

If we want to identify the limit of a bunch of colliding double points with some $Z_n(l_1,\dots,l_t)$, our next task is to study such schemes. First we compute the multiplicity.

\begin{lem}\label{lem:molteplicitÃ _unione}
	Let $R=\A^{n-1}$ be a general hyperplane in $\A^n$, and $p_i:=l_i\cap R$. Define $B:=\{p_1,\dots,p_t\}$ and set
	\[k:=\min\{m\in\N\mid H^0\I_{B,R}(m)\neq 0\}.\]
	Then $\mul Z_n(l_1,\dots,l_t)=\min(4,k)$.
	\begin{proof}
		First note that $\mul Z_n(l_1,\dots,l_t)$ is nondecreasing with respect to $t$. Moreover, $\mul Z_n(l_1,\dots,l_t)\le 4$ by construction. Indeed, once multiplicity 4 is reached, the restriction to any line has degree at least 4, so by adding another $S^4_i$ we do not change anything. Now let $D\subset R$ be a degree $m$ divisor containing $p_1,\dots,p_t$. The cone $C$ over $D$ with vertex the origin is a degree $m$ divisor in $\A^n$ containing $l_1,\dots,l_t$ and therefore $C\supset S^4_1\cup\ldots\cup S^4_t$. Hence the ideal of $Z_n(l_1,\dots,l_t)$ contains a generator of degree $m$ and so $\mul Z_n(l_1,\dots,l_t)\le m$. This implies $\mul Z_n(l_1,\dots,l_t)\le\min(4,k)$.
		
		On the other hand, if $\mul Z_n(l_1,\dots,l_t)=4\ge \min(4,k)$, then there is nothing else to prove. Suppose that $m:=\mul Z_n(l_1,\dots,l_t)\in\{1,2,3\}$. Then $Z_n(l_1,\dots,l_t)$ is contained in a degree $m$ divisor $C\subset\A^n$. Since $C$ has an $m$-tuple point, it is a cone. Moreover the restriction of $Z_n(l_1,\dots,l_t)$ to each $l_i$ has degree $4>m$ so $C$ contains each $l_i$, and in particular $C_{|R}$ is a degree $m$ divisor in $R$ containing $p_1,\dots,p_t$.
	\end{proof}
\end{lem}

\begin{cor}\label{cor:molteplicitÃ _unione}
	Let $t\in\N$ and let $R=\A^{n-1}$ be a general hyperplane in $\A^n$. Set
	\[k:=\min\{m\in\N\mid \h^0\oo_R(m)>t\}.\]
	If $l_1,\dots,l_t$ are general lines, then $\mul Z_{n,t}=\min(4,k)$.
	\begin{proof}
		Apply Lemma \ref{lem:molteplicitÃ _unione} in the case $p_1,\dots,p_t\in R$ are general.
	\end{proof}
\end{cor}

Now we want to determine the length of $Z_n(l_1,\dots,l_t)$. The next Lemma provides a way to compute it inductively.

\begin{lem}\label{lem:grado_unione} Let $n\ge 2$. Then
	\begin{enumerate}
		\item $\deg Z_n(l_1)=4$,
		\item $\deg Z_n( l_1,\dots,l_t,l_{t+1}) =\deg Z_n(l_1,\dots,l_t)+4-\deg ( Z_n(l_1,\dots,l_t)_{|l_{t+1}}) $,
		\item $\deg Z_{n,t+1}=\deg Z_{n,t}+4-\mul Z_{n,t}$.
	\end{enumerate}
	\begin{proof}
		\begin{enumerate}
			\item The degree of $Z_n(l_1)=S^4_1$ does not depend on its embedding. Regarding $S^4_1$ as a divisor in $l_1=\p^1$, it has degree 4 by construction.
			\item Let $\mu=\deg ( Z_n(l_1,\dots,l_t)_{|l_{t+1}})$. Of course $Z_n(l_1,\dots,l_t)\supset S_{t+1}^\mu$, so
			\[Z_n(l_1,\ldots,l_t)=S_1^4\cup\ldots\cup S_t^4=S_1^4\cup\ldots\cup S_t^4\cup S_{t+1}^\mu.\]
			Hence the difference $\deg Z_n( l_1,\dots,l_t,l_{t+1})-\deg Z_n(l_1,\dots,l_t)$ coincides with the difference $\deg S_{t+1}^4-\deg S_{t+1}^\mu=4-\mu$.
			\item When $l_1,\dots,l_t,l_{t+1}$ are general, the restriction of $Z_{n,t}$ to $l_{t+1}$ has degree equal to $\mul Z_{n,t}$, so it is enough to apply (2).\qedhere
		\end{enumerate}
	\end{proof}
\end{lem}

Corollary \ref{cor:molteplicitÃ _unione} and Lemma \ref{lem:grado_unione} allow us to compute multiplicity and degree of the scheme $Z_{n,t}$ for every $n$ and $t$.
Now we consider what happens when the lines are not general. Namely, we are interested in the configuration described in Remark \ref{rmk:nongenerici}.

\begin{dfn}
	Let $\{l_{ij}\mid 1\le i<j\le m\}$ be a set of $\binom{m}{2}$ lines in $\A^n$ meeting at the origin, such that $l_{ab}$, $l_{bc}$ and $l_{ac}$ lie on the same plane for every $1\le a<b<c\le m$. Define $$\tilde{Z}_{n,\binom{m}{2}}:=Z_n(l_{ij}\mid 1\le i<j\le m).$$
\end{dfn}

\begin{rmk}\label{rmk:trivialtilda} Let $n,m\ge 2$. We start with the following simple observations.
	\begin{enumerate}
		\item $Z_{2,\binom{m}{2}}=\tilde{Z}_{2,\binom{m}{2}}$.
		\item $\tilde{Z}_{n,1}=Z_{2,1}$ and $\tilde{Z}_{n,3}=Z_{2,3}$.
		\item More generally, if $n\ge m$ then $\left\langle l_1,\dots,l_{\binom{m}{2}}\right\rangle\subseteq\A^{m-1}$. Thus $\mul \tilde{Z}_{n,\binom{m}{2}}=1$ and $\tilde{Z}_{n,\binom{m}{2}}=\tilde{Z}_{m-1,\binom{m}{2}}$.
	\end{enumerate}
\end{rmk}

We are now ready to compute multiplicity and degree of $\tilde{Z}_{n,\binom{m}{2}}$. By Remark \ref{rmk:trivialtilda}, we know multiplicity and degree of $\tilde{Z}_{2,\binom{m}{2}}$ from Lemma \ref{lem:grado_unione}. Now we tackle the cases $n=3$ and $n=4$.

\begin{es}
	The next table shows the values of $\deg\tilde{Z}_{3,\binom{m}{2}}$ and $\mul\tilde{Z}_{3,\binom{m}{2}}$.
	\begin{center}
		\begin{tabular}{|c|c|c|c|}
			\hline
			$m$	&	$t$	&	$\deg \tilde{Z}_{3,t}$	&	$\mul\tilde{Z}_{3,t}$	\bigstrut\\
			\hline
			2	&	1	&	4	&	1\\
			3	&	3	&	9	&	1\\
			4	&	6	&	16	&	3\\
			$m\ge 5$	&	$t\ge 10$	&	20&	4	\\
			\hline
		\end{tabular}
	\end{center}
	Degrees and multiplicities of $\tilde{Z}_{4,\binom{m}{2}}$ are presented in the following one.
	\begin{center}
		\begin{tabular}{|c|c|c|c|}
			\hline
			$m$	&	$t$	&	$\deg \tilde{Z}_{4,t}$	&	$\mul\tilde{Z}_{4,t}$	\bigstrut\\
			\hline
			2	&	1	&	4	&	1\\
			3	&	3	&	9	&	1\\
			4	&	6	&	16	&	1\\
			5	&	10	&	25	&	3\\
			6	&	15	&	30	&	3\\
			7	&	21	&	34	&	3\\
			$m\ge 8$	&	$t\ge 28$	&	35	&	4\\
			\hline
		\end{tabular}
	\end{center}
	In order to compute the multiplicities, it is enough to apply Remark \ref{rmk:trivialtilda} and Lemma \ref{lem:molteplicitÃ _unione}, together with Lemma \ref{lem:indip}. After that, Lemma \ref{lem:grado_unione} allows us to compute the degree. We only have to pay attention for $(n,m)=(4,7)$. Indeed,
		this is an exception of Theorem \ref{thm:AH}, and we already considered it in Example \ref{es:7doppiinp4}.
	\end{es}

If we look at $\tilde{Z}_{3,6}$ and $\tilde{Z}_{3,10}$, we see that their multiplicities and degrees are consistent with the cases of 4 and 5 collapsing double points in $\A^3$. In the same way, the numbers we found about $\tilde{Z}_{4,10}$ and $\tilde{Z}_{4,15}$ are consistent with the case of 5 and 6 colliding double points in $\A^4$.
We will try now to find a general statement about the degree and the multiplicity of $\tilde{Z}_{n,\binom{m}{2}}$. The situation is easy when $m\le n$.

\begin{pro}\label{pro:emmebasso} If $3\le m\le n$, then $\mul\tilde{Z}_{n,\binom{m}{2}}=1$ and $\deg\tilde{Z}_{n,\binom{m}{2}}
	=m^2.$
	\begin{proof}
		If $m\le n$, then $\mul\tilde{Z}_{n,\binom{m}{2}}=1$ by Remark \ref{rmk:trivialtilda}.
		We prove the statement about the degree by induction on $m$. We saw that $\deg\tilde{Z}_{n,3}=9$. Let us assume $\deg\tilde{Z}_{n,\binom{m}{2}}=m^2$ and let us compute $\deg\tilde{Z}_{n,\binom{m+1}{2}}$. $\tilde{Z}_{n,\binom{m+1}{2}}$ is obtained from $\tilde{Z}_{n,\binom{m}{2}}\subset\A^m=H$ by adding $S^4_{1,m+1},\dots,S^4_{m,m+1}$. Observe that $S^4_{1,m+1}\not\subset H$, so it increases the degree by 3; the resulting scheme is contained in some $W=\A^{m+1}$, and by adding $S^4_{1,m+1},\dots,S^4_{m-1,m+1}$, we remain inside $W$. As a subscheme of $W$, $\tilde{Z}_{n,\binom{m+1}{2}}$ has multiplicity 2, because there are only $\binom{m}{2}+m-1<\h^0\oo_{\A^{n-1}}(2)$ lines. Even if they are in special position, they are general for quadrics by Lemma \ref{lem:indip}, so each new addition of $S^4_{2,m+1},\dots,S^4_{m,m+1}$ increases the degree by $4-2=2$. Hence
		\[\deg\tilde{Z}_{n,\binom{m+1}{2}}=\deg\tilde{Z}_{n,\binom{m}{2}}+3+2(m-1)=m^2+2m+1=(m+1)^2. 
		\qedhere\]
\end{proof}\end{pro}

Before we move to the more interesting case $m>n\ge 5$, we need some technical results. We already observed that Lemma \ref{lem:indip} does not hold in the case of more than $n+2$ points in $\p^n$, so our next goal is to understand what happens with larger numbers of points. 

\begin{lem}	\label{lem:quantecodizionidanno} For $k\in\N$, define
	$$n_k:=\min\left\lbrace t\ge 2\mid \frac{\binom{t+3}{3}}{t+1}-t>k\right\rbrace.$$
	For every $n\ge n_k$ and every $r\in\N$, let $A_r:=\{a_1,\dots,a_{r}\}\subset\p^{n}$ be a set of $r$ general points, and let $R\subset\p^n$ be a hyperplane such that $A_r\cap R=\varnothing$. Let $p_{ij}:=\langle a_i,a_j\rangle\cap R$ and 
	$$B_r:=\{p_{ij}\mid 1\le i<j\le r\}.$$	
	Assume that $B_{n_k+k}$ imposes $\binom{n_k+k}{2}-\binom{k-1}{2}$ independent conditions to cubics of $R$. Then $B_{n+k}$ impose exactly $\binom{n+k}{2}-\binom{k-1}{2}$ independent conditions to cubics of $R$ for every $n\ge n_k$.
	\begin{proof}
		We prove the statement by induction on $n\ge n_k$. The first step of induction is granted by hypothesis, so we suppose that $n>n_k$. In order to lighten the notation, throughout this proof we will write $A$ and $B$ instead of $A_{n+k}$ and $B_{n+k}$.
		Specialize $a_1,\dots,a_{n+k-1}$ on an hyperplane $W\subset\p^{n}$. Define $$B_1:=\{p_{ij}\mid 1\le i<j\le n+k-1\}\mbox{ and } B_2:=\{p_{1,n+k},\dots,p_{n+k-1,n+k}\}.$$ Let $H:=W\cap R=\p^{n-2}$. Castelnuovo exact sequence reads
		\[0\to\I_{B_2,R}(2)\to\I_{B,R}(3)\to\I_{B_1,H}(3)\to 0.\]
		First observe that the points of $B_2$ are general on $R$, so $$\h^0\I_{B_2,R}(2)=\binom{n+1}{2}-(n+k-1)\mbox{ and }\h^1\I_{B_2,R}(2)=0.$$
		Now we want to compute the dimension of the right hand side of the sequence.
		Note that $A_1:=\{a_1,\dots,a_{n+k-1}\}$ is a set of general points in $W=\p^{n-1}$, $H$ is a hyperplane of $W$ with $A_1\cap H=\vu$ and $B_1=\{\langle a_i,a_j\rangle\cap H\mid 1\le i<j\le n+k-1\}$, so by induction hypothesis $$\h^0\I_{B_1,H}(3)=\binom{n+1}{3}-\binom{n+k-1}{2}+\binom{k-1}{2}.$$
		Therefore
		\begin{align}
			\h^0\I_{B,R}(3)&=\h^0\I_{B_2,R}(2)+\h^0\I_{B_1,H}(3)\nonumber\\
			&=\binom{n+1}{2}-(n+k-1)+\binom{n+1}{3}-\binom{n+k-1}{2}+\binom{k-1}{2}\nonumber\\
			&=\binom{n+2}{3}-\binom{n+k}{2}+\binom{k-1}{2}.\nonumber
		\end{align}
		Since the points of $B$ impose $\binom{n+k}{2}-\binom{k-1}{2}$ conditions in this specialized configuration, they impose at least $\binom{n+k}{2}-\binom{k-1}{2}$ conditions in the original configuration. We already noticed they cannot impose more than $\binom{n+k}{2}-\binom{k-1}{2}$ conditions.
	\end{proof}
\end{lem}

Lemma \ref{lem:quantecodizionidanno} provides an inductive way to prove that $B$ imposes the suitable number of conditions on cubics of $R$. However, in order to apply it we need the first step of induction for every $k$. While we are not able to prove this first step in general, we believe this is the right way to compute the number of independent conditions imposed by $B$.

\begin{con}\label{con:n+k}
	Assume $0\le k<\frac{\binom{n+3}{3}}{n+1}-n$. Let $A:=\{a_1,\dots,a_{n+k}\}$ be a set
	of $n+k$ general points in $\p^n$ and $R$ a hyperplane such that $A\cap R=\varnothing$. Let
	$$B:=\{\langle a_i,a_j\rangle\cap R\mid 1\le i<j\le n+k\}.$$
	Then the points of $B$ impose exactly $\binom{n+k}{2}-\binom{k-1}{2}$ independent conditions to cubics of $R$.
\end{con}

By applying Lemma \ref{lem:quantecodizionidanno}, it is easy to prove that Conjecture \ref{con:n+k} holds for $k\in\{0,1,2\}$, and in this way we recover some of the results of Lemma \ref{lem:indip}. Moreover, a software computation allows us to prove the first step for $k\le 4$ as well.

\begin{rmk}
	Assume Conjecture \ref{con:n+k} is true. Then we have a way to compute degree and multiplicity of $\tilde{Z}_{n,\binom{m}{2}}$. Indeed, assume that
	\begin{equation}\label{eq:range k}
	1\le k<\frac{\binom{n+3}{3}}{n+1}-n\mbox{ and } (n,k)\neq (4,3).
	\end{equation}
	On one hand, $\mult\tilde{Z}_{n,\binom{n+k}{2}}\ge 3$ by Lemma \ref{lem:molteplicitÃ _unione}. On the other hand, we observed in Remark \ref{rmk:graziekris} that $\tilde{Z}_{n,\binom{n+k}{2}}$ is a subscheme of the limit of $n+k$ collapsing double points, which has multiplicity 3 because $k$ is in the range (\ref{eq:range k}). Hence $\mul\tilde{Z}_{n,\binom{n+k}{2}}=3$. Then, by Lemma \ref{lem:degree of a fat point with infinitely near simple points}, its degree is	
	\[\deg\tilde{Z}_{n,\binom{n+k}{2}}=\binom{n+2}{n}+\binom{n+k}{2}-\binom{k-1}{2}=(n+1)(n+k),\]
	where the last equality can be proven by induction on $k$.
	Therefore, under this assumption, the limit of $n+k$ collapsing double points in $\A^n$ is $\tilde{Z}_{n,\binom{n+k}{2}}$. Since we know that Conjecture \ref{con:n+k} holds for small values of $k$, this improves Propositions \ref{collision1} and \ref{collision2}.
	However, this approach only works in the range (\ref{eq:range k}).	When $k\le 0$, the limit scheme has multiplicity 2. As we already pointed out, the linear system $ L_{n,2}(2^{n+k})$ has nonreduced base locus, and this makes it difficult to understand the first order neighbourhood of the limit point. On the other hand, when $k+n\ge\frac{\binom{n+3}{3}}{n+1}$, the limit scheme has multiplicity at least 4 and the base locus may not give us information. It is enough to consider $(n,k)=(3,3)$ to bump into the linear system $ L_{3,4}(2^6)$, which has no base locus outside the imposed singularities. Our work on infinitely near points gives us no clue in this type of cases.
\end{rmk}

One could argue in a similar way with higher multiplicities and hope to find other cases in which there are base lines. For instance, we could work with triple points, and we know that the lines joining a pair of triple points are in the base locus of quintics.
\begin{es}\label{es:sporadici1} 
	Consider 5 collapsing 5-ple points in $\A^3$. Since $ L_{3,8}(5^4)$ is empty, the limit has multiplicity 9 by Proposition \ref{pro:molteplicità_limite}. The base locus of $ L_{3,9}(5^4)$ consists of 10 lines which cut 10 simple points on the exceptional divisor $R$. They are in the special position described by Remark \ref{rmk:nongenerici}, but still they impose independent conditions on 9-ics of $R$ by Lemma \ref{lem:indip}. The limit scheme has degree 175 and contains a 9-tuple point with 10 infinitely near simple points. Since the latter has degree 175 by Lemma \ref{lem:degree of a fat point with infinitely near simple points}, they coincide by Lemma \ref{stessogrado}.
\end{es}
Unfortunately, this strategy works only if we know the degree of the linear system we are dealing with. By Proposition \ref{pro:molteplicità_limite}, this is equivalent to compute the smallest degree of a divisor in $\p^n$ containing a bunch of general multiple points. The answer is unknown in general. 
Moreover, Corollary \ref{cor:condizioni imposte dai punti infinitamente vicini} does not work when the infinitely near points have multiplicity greater than 1.

It is also worth to mention that, given a scheme $X\subset\A^n$ made by a triple point with $t$ tangent directions, in general we cannot produce $X$ as a limit of double points. Indeed, first we need that $t=\binom{n+k}{n}$ for some $k$ in the range (\ref{eq:range k}). Moreover, the tangent directions have to be in the special position described in Remark \ref{rmk:nongenerici}. It is legitimate to wonder if there are more conditions to be met in order to express $X$ as a limit of double points. In other words, can we lift $X$ to a bunch of double points in such a way that $X$ is the limit of those colliding points, under the previous assumptions?
We will now give a positive answer to this question.
Remark \ref{rmk:nongenerici} describes the configurations of the points in the exceptional divisor and suggests the following definition.

\begin{dfn} Let $n\ge 2$ and $t\ge 3$. Define
	\[W_{n,t}=\left\lbrace (x_{ij})_{1\le i<j\le t}\in (\p^n)^{\binom{t}{2}}\mid x_{bc}\in\langle x_{ab}, x_{ac}\rangle\ \forall\ 1\le a<b<c\le t\right\rbrace .\]
	If we look at $R=\p^n$ as a general hyperplane in $\p^{n+1}$, there is a rational map
	\[\pi_{n,t}:\left( \p^{n+1}\right) ^t\dashrightarrow W_{n,t}\subset (\p^n)^{\binom{t}{2}}\]
	defined by sending $(p_1,\dots,p_t)$ to $(x_{ij})_{1\le i<j\le t}$, where $x_{ij}$ is the intersection of the line $\langle p_i,p_j\rangle$ with $R$.
\end{dfn}

For $1\le k\le 4$, we know that the limit of $n+k$ double points in $\p^n$ is a triple point with $\binom{n+k}{2}$ infinitely near simple points. The simple points form a $\binom{n+k}{2}$-tuple $(x_{ij})_{1\le i<j\le n+k}\in W_{n,n+k}$. We want to understand whether such schemes can be obtained as limits of double points. This is equivalent to ask if $\pi_{n,n+k}$ is dominant. We will prove the following result.

\begin{thm}\label{thm:lift}
	$\pi_{n,t}
	$ is dominant for every $n\ge 2$ and every $t\ge 3$. The general fiber has dimension $n+2$.
\end{thm}

Let us start with some simple observations.

\begin{obs}\label{obs:configuraz}
	\begin{enumerate}
		\item We have $\dim W_{n,t}=n(t-1)+t-2$. Indeed, one can choose freely $t-1$ general points $x_{12},\dots,x_{1t}\in \p^n$. Then, for $i\in\{3,\dots,t\}$, it is possible to choose the $t-2$ points $x_{2i}$ general on $\langle x_{12},x_{1i}\rangle$. After that, for $3\le j<k\le t$, the other points $x_{jk}$ are defined by $\langle x_{1j},x_{1k}\rangle\cap\langle x_{2j},x_{2k}\rangle$.
		\item Assume that $t\ge 4$ and let $(x_{ij})_{1\le i<j\le t}\in W_{n,t}$. For $1\le a<b<c\le t$, let $l_{abc}$ be the line containing $x_{ab},x_{ac},x_{bc}$. Note that $l_{abc}$ and $l_{bcd}$ meet at $x_{bc}$, so they span a plane containing $l_{acd}$ and $l_{abd}$ as well. This plane therefore passes through the 6 points $\{x_{ij}\mid i,j\in\{a,b,c,d\},i<j\}$. By the same argument, if $t\ge m$ then for every choice of $m$ indexes $1\le i_1<\ldots<i_m\le t$ the $\binom{m}{2}$ points $\{x_{ij}\mid i,j\in\{i_1,\dots,i_m\}\}$ lie on the same $\p^{m-2}$.
		\item In particular, if $t\le n+1$ then $p_1,\dots,p_t\in\p^{n+1}$ lie on a linear subspace $L=\p^{t-1}$. Hence the $\binom{t}{2}$ points $\langle p_i,p_j\rangle\cap R$ all lie on $L\cap R=\p^{t-2}$. 
		Then $W_{n,t}=W_{t-2,t}$, and $\pi_{n,t}$ restricts to $\pi_{t-2,t}:L^t=(\p^{t-1})^t\dashrightarrow W_{t-2,t}$. For this reason, from now on we will assume $t\ge n+2$.
	\end{enumerate}
\end{obs}

The next Lemma is the first step towards the proof of Theorem \ref{thm:lift}.

\begin{lem}\label{lem:sollevamento1}
	$\pi_{n,n+2}:(\p^{n+1})^{n+2}\dashrightarrow W_{n,n+2}$ is dominant for every $n\ge 2$. The general fiber has dimension $n+2$.
	\begin{proof}
		Let $x=(x_{ij})_{1\le i<j\le n+2}\in W_{n,n+2}$ be general. 
		For $i\in\{1,\dots,n+2\}$, let $L_i=\Span{x_{jk}\mid j,k\neq i}$ be the dimension $n-1$ linear subspace of $R=\p^n$ obtained by choosing all indexes except $i$.
		Let $\Pi_i\subset\p^{n+1}$ be a general hyperplane containing $L_i$. For $j\in\{1,\dots,n+2\}$, define the point
		\[p_j:=\bigcap_{i\neq j}\Pi_i.\]
		If $k,h\in\{1,\dots,n+2\}$ and $h\neq k$, then $p_{k}$ and $p_{h}$ are distinct points of the line $\bigcap_{i\neq k,h}\Pi_i$, so \[\langle p_h,p_k\rangle\cap R=\bigcap_{i\neq k,h}\Pi_i\cap R=\bigcap_{i\neq k,h}L_i,\]
		which is one of the $x_{ij}$'s. Then, up to reorder, $(p_1,\dots,p_{n+2})$ is a preimage of $(x_{ij})_{1\le i<j\le n+2}$.
		
		To determine the dimension of the general fiber, we can either note that for each of the $n+2$ points $p_i$ we chose a hyperplane $\Pi_i$ in the pencil of those containing $L_i$, or we can compute the difference $\dim (\p^{n+1})^{n+2}-\dim W_{n,n+2}$.
	\end{proof}
\end{lem}

One could give the definition of $W_{1,t}$ and $\pi_{1,t}$ as well. However, we are computing limits under the assumption that $n\ge 2$. Moreover $W_{1,t}=(\p^1)^{\binom{t}{2}}$, so the case $n=1$ is not very interesting for us.

We are now ready to prove the result we claimed.
\begin{proof}[Proof of Theorem \ref{thm:lift}]
	As we noticed in Observation \ref{obs:configuraz}, we may assume that $t\ge n+2$. We argue by induction on $t$.
	The case $t=n+2$ is the content of Lemma \ref{lem:sollevamento1}, so we focus on the case $t>n+2$.
	
	Let $(x_{ij})_{1\le i<j\le t}\in W_{n,t}$ be general. By induction hypothesis exist $t-1$ general points $p_1,\dots,p_{t-1}\in\p^{n+1}$ such that $\langle p_i,p_j\rangle\cap R=x_{ij}$. Define
	$$p_t:=\langle p_1,x_{1t}\rangle\cap\langle p_2,x_{2t}\rangle.$$ We have to make sure that $\langle p_i,p_t\rangle$ meets $R$ at $x_{it}$ for every $i\in\{3,\dots,t-1\}$. Observe that
	\[\langle x_{1i},x_{1t}\rangle=\langle p_1,x_{1i},x_{1t}\rangle \cap R=\langle p_1,p_i,p_t\rangle\cap R,\]
	because $p_t\in\langle p_1,x_{1t}\rangle$ by construction. Hence
	\begin{align}
		\langle p_i,p_t\rangle \cap R&=\left( \langle p_1,p_i,p_t\rangle\cap \langle p_2,p_i,p_t\rangle\right) \cap R\nonumber\\
		&=\left( \langle p_1,p_i,p_t\rangle\cap R\right) \cap\left( \langle p_2,p_i,p_t\rangle\cap R\right)\nonumber\\
		&=\langle x_{1i},x_{1t}\rangle\cap \langle x_{2i},x_{2t}\rangle=x_{it}.\nonumber
	\end{align}
	The general fiber has dimension $\dim\left( \p^{n+1}\right) ^t-\dim W_{n,t}=n+2$.
\end{proof}

In terms of collision, this means that if $t\in\{n+1,\dots,n+4\}$, then every subscheme of $\p^n$ made by a triple point with $\binom{t}{2}$ infinitely near simple points $x_{ij}$ such that $(x_{ij})_{1\le i<j\le t}$ is a general point of $W_{n,t}$ can be obtained as a limit of $t$ collapsing double points in $\p^{n+1}$. If Conjecture \ref{con:n+k} is true, the same holds for the collision of $t$ double points, where
\[n+1\le t<\frac{\binom{n+3}{3}}{n+1}.\]

\section{Higher multiplicities}\label{sec:othercollisions}\label{sec:lowdimension}

Degenerations are widely used in interpolation theory to compute the dimension of linear systems. The most studied cases are dimension 2 and 3, where there are conjectures about the reasons why a linear system is special. For $n\in\{2,3\}$, all known special linear systems $L_{n,d}(m_1,\dots,m_r)$ have a base locus containing a particular variety. Roughly speaking, what those conjectures state is that the only geometric reason for a linear system to be special is the existence of such a \textit{special effect variety} in its base locus. The precise definition of special effect varieties can be found in 
\cite{Boc1}. Some examples of special effect varieties are known (see \cite{BDP1} and \cite{BDP2}) and the hard problem is to classify all of them.
We will not look into special effect varieties, but we will apply some of the known results in interpolation theory to describe limits of colliding multiple points.
As a first example, we can easily extend Proposition \ref{pro:doppisumultiplo} to higher multiplicity.

\begin{pro}\label{pro:n-pli_su_m-plo}\label{omogenei_su_multiplo_P^3} Let $l,m\in\N$, let $n\in\{2,3\}$ and define
	$$h:=\frac{\binom{m+n-1}{n}}{\binom{l+n-1}{n}}.$$
	Assume that $h\in\N$.
	\begin{enumerate}
		\item Consider $n=2$. If $l,m\le 42$ and $h\ge 10$, then the limit of $h$ collapsing $l$-tuple points in $\A^2$ is an $m$-tuple point.
		\item Consider $n=3$. If $l\le 5$ and $m\ge 2l+1$, then the limit of $h$ collapsing $l$-tuple points in $\A^3$ is an $m$-tuple point.
	\end{enumerate}
	\begin{proof}
		\begin{enumerate}
			\item Observe that $ L_{2,m-1}(l^h)$ is a system of plane curves with $h\ge 10$ fixed points of the same multiplicity, hence \cite[Conjecture 5.10]{Cil} predicts that it is non-special. For $l,m\le 42$ the conjecture is proven true in \cite[Theorem 32]{Du}. Now we conclude by Lemma \ref{lem:grassisugrasso}.
			\item By Lemma \ref{lem:grassisugrasso}, it is enough to prove that $ L_{3,m-1}(l^h)$ is non-special. If $m\ge 12$, we conclude by \cite[Theorem 1]{BBCS}. Let us check the remaining cases.
	If $l=5$, we only have to check $m=11$, but in this case $h\notin\N$. If $l=4$, we have to consider $m\in\{9,10,11\}$. Only $m=10$ gives an integer $h$, and $L_{3,9}(4^{11})$ is non-special by \cite[Theorem 14]{BB}. Finally, if $l=3$ we have to check $7\le m\le 11$. If $m\in\{7,9,11\}$ then $h\notin\N$. A software computation shows that $L_{3,7}(3^{12})$ and $L_{3,9}(3^{22})$ are non-special.\qedhere
	\end{enumerate}	\end{proof}
\end{pro}



Observe that in point (2) the assumption $m\ge 2l+1$ is necessary. For instance, consider $l=4$ and $m=8$. Then $h=6$ and we deal with $ L_{3,7}(4^6)$. We can study it by applying a standard Cremona transformation in $\p^3$. By \cite[Proposition 2.1]{LU},
$$ L_{3,7}(4^6)\cong L_{3,5}(4^2,2^4)\cong L_{3,3}(2^4)$$
is not empty, so $\mul Z_0=7<m$.


Up to now, we mostly considered collisions of points with the same multiplicities, but of course there are many other cases in which we can determine the limit $Z_0$. 
%
The next two Propositions will deal with the case of a fat point colliding together with a bunch of low multiplicity points. 

\begin{pro}\label{pro:grossoesemplici}
	Let $m,n\ge 2$ and let $h:=\binom{n+m-1}{n}$. Then the limit of $h$ simple points and a point of multiplicity $m$ colliding in $\A^n$ is a $(m+1)$-tuple point.
	\begin{proof}
		Since $m\ge 2$, $ L_{n,m}(m,1^h)=0$. On the other hand, $ L_{n,m+1}(m,1^h)\neq 0$, so $\mul Z_0=m+1$ by Proposition \ref{pro:molteplicità_limite}. To conclude, observe that a $(m+1)$-tuple point has degree $\binom{n+m}{m}=\deg Z_1$, so $Z_0$ is a $(m+1)$-tuple point.
	\end{proof}
\end{pro}	

\begin{pro}
	Let $m,n\ge 3$ and $(m,n)\notin\{(4,3),(3,5)\}$. Suppose that $h:=\frac{\binom{m+n-1}{n-1}}{n}\in\N$. Then the limit of $h$ double points and a point of multiplicity $m$ colliding in $\A^n$ is a $(m+1)$-tuple point with $h$ infinitely near general simple points.
	\begin{proof}
		By hypothesis \begin{align}
			\vdim L_{n,m+1}(m,2^h) &=\binom{n+m+1}{n}-\binom{n+m-1}{n}-h(n+1)\nonumber\\
			=&\binom{n+m+1}{n}-\binom{n+m-1}{n}-\frac{\binom{m+n-1}{n-1}}{n}(n+1)\nonumber\\
			=&\binom{n+m+1}{n}-\binom{n+m-1}{n}-\binom{m+n-1}{n-1}-\frac{\binom{m+n-1}{n-1}}{n}\nonumber\\
			=&\frac{(n+m+1)!}{n!(m+1)!}-\frac{(n+m-1)!}{n!(m-1)!}-\frac{(m+n-1)!}{(n-1)!m!}-\frac{(m+n-1)!}{m!n!}\nonumber\\
			=&\frac{(n+m-1)!}{(n-1)!(m-1)!}\left[\frac{(n+m+1)(n+m)}{(m+1)mn}-\frac{1}{n}-\frac{1}{m}-\frac{1}{mn}\right]\nonumber\\
			=&\frac{(n+m-1)!}{(n-1)!(m-1)!}\left[\frac{n^2 + mn - m - 1}{(m^2 + m)n}\right]>0,\nonumber
		\end{align}
		hence $ L_{n,m+1}(m,2^h)$ is not empty. On the other hand,
		\[\binom{n+m-1}{n-1}-hn=\binom{n+m-1}{n-1}-\binom{m+n-1}{n-1}=0,\]
		so $ L_{n,m}(m,2^h)\cong L_{n-1,m}(2^h)$ is expected to be empty. The latter is non-special by Theorem \ref{thm:AH}, so $\mul Z_0=m+1$ by Proposition \ref{pro:molteplicità_limite}. The $h$ general lines joining the $m$-tuple point and one of the double points are contained in the base locus of $ L_{n,m+1}(m,2^h)$, and they cut $h$ simple points on $R$. The candidate limit scheme is a $(m+1)$-tuple point with $h$ infinitely near general simple points, which by Lemma \ref{lem:degree of a fat point with infinitely near simple points} has length $\binom{n+m-1}{n}+h=\deg Z_1$.
	\end{proof}
\end{pro}

We now focus on $n=3$. Recall that 8 is the maximum $r$ such that we know the full classification of special linear systems $ L_{3,d}(m_1,\dots,m_r)$, see \cite[Theorem 5.3]{DL}. The following Proposition will be useful to get some results beyond this bound.

\begin{pro}\label{pro:limitediottograssi}
	The limit of the collision of $8$ $m$-tuple points and $m+1$ simple points in $\A^3$ is a point of multiplicity $2m+1$.
	\begin{proof}
		First we check that
		\[8\binom{m+2}{3}+m+1=\binom{2m+3}{3}.\]
		By Lemma \ref{lem:grassisugrasso}, it is enough to prove that $ L_{3,2m}(m^8,1^{m+1})$ is non-special. Since general simple points always give independent conditions, it suffices to show that $ L_{3,2m}(m^8)$ is non-special. The latter is true by \cite[Theorem 5.3]{DL}.
	\end{proof}
\end{pro}

	In this Section, our arguments to describe limits of collisions of fat points rely on known results on non-special linear systems of $\p^2$ and $\p^3$. With a similar approach, other results can be applied in the same way to prove more statements about collisions. Examples on $\p^3$ include \cite[Theorem 5.8]{BDP2} and the aforementioned \cite[Theorem 5.3]{DL}. On $\p^2$ there is \cite[Theorem 34]{DJ}, as well as the results contained in \cite{Roé} and in the survey \cite{Cil}.
	While most of our knowledge of special linear systems is concentrated in low dimensional varieties, there is also something we can say about any $\p^n$. As an example, there are the results contained in \cite{BDP1}.

\begin{pro}\label{pro:collisionetripliinp4}
	The limit of the collision of $6$ triple points and $36$ simple points in $\A^4$ is a point of multiplicity $6$.
	\begin{proof} The proof works as in Proposition \ref{pro:limitediottograssi}. We check that
		\[6\binom{6}{4}+36=\binom{9}{4}.\]
		By Lemma \ref{lem:grassisugrasso}, it is enough to prove that $ L_{4,5}(3^6,1^{36})$ is non-special. Again, general simple points always give independent conditions, so we only have to show that $ L_{4,5}(3^6)$ is non-special. The latter is true by \cite[Corollary 4.8]{BDP1}.
	\end{proof}
\end{pro}

Next Proposition has a slightly different flavour. It states that, up to adding a bunch of simple points, we can always turn two fat points into a unique fat point.

\begin{pro}\label{pro:duesumultiplo}
	Let $m_1,m_2\in\N$. Then exist $h,m\in\N$, depending on $n,m_1,m_2$, such that the limit of two points of multiplicity $m_1$ and $m_2$ and $h$ simple points in $\A^n$ is an $m$-tuple point.
	\begin{proof}
		Define $m:=m(n,m_1,m_2):=m_1+m_2+1$ and
		\begin{equation*}\label{eq:numeropuntisemplici}
			h:=h(n,m_1,\dots,m_s):=\binom{m+n}{n}-\binom{m_1+n-1}{n}-\binom{m_2+n-1}{n}-1.
		\end{equation*}
		By construction $\vdim L_{n,m}(m_1,m_2,1^h)\ge 0$, hence $ L_{n,m}(m_1,m_2,1^h)$ is not empty. Since an $m$-tuple point has degree equal to the length of the starting scheme, it is enough to show that $ L_{n,m-1}(m_1,m_2,1^h)$ is non-special, and therefore empty. Since the $h$ simple points always give independent conditions, it suffices to prove that $ L_{n,m-1}(m_1,m_2)$ is non-special. By \cite[Corollary 4.8]{BDP1}, such system is linearly non-special, so in order to conclude we just need to observe that there are no base linear cycles.
	\end{proof}
\end{pro}

\section{Applications to interpolation theory}\label{section:interpolation}

A first important application of collisions to interpolation theory is \cite{GM}, where a collision of double points allows the authors to solve a long-standing problem about Waring decompositions of polynomials.
Another application is \cite{EvainSHGH}, where a suitable collision of fat points in $\p^2$ is used to prove that Segre's conjecture (see \cite[Conjecture 4.1]{Cil}) holds for an infinite family of linear systems of plane curves.
This approach can be pushed further. In Section \ref{sec:lowdimension} we used known results in interpolation theory to provide clues about what the limit is. It is just fair to try to return the favour, using the limits we constructed as tools to specialize linear systems and prove that they are non-special.

We begin 
 with our contribution to Laface-Ugaglia conjecture (see \cite[Conjecture 4.1]{LU} and \cite[Conjecture 5.1]{BDP2}). When all multiplicities are the same, the conjecture predicts that $ L_{3,d}(m^r)$ is non-special whenever
 \[(d+1)^2>9\binom{m+1}{2}\mbox{ and }d\ge 2m-1.\]
 The most restrictive of these two conditions is the first one, which reads
	\[d>-1+3\sqrt{\frac{m^2+m}{2}}=\frac{3}{\sqrt{2}}m+o(m).\]
	As we stated in Theorem \ref{thm:LUcondgrande}, for systems with at most 15 base points we are able to prove that the conjecture holds under the stronger assumption $d\ge 3m$.

	\begin{proof}[Proof of Theorem \ref{thm:LUcondgrande}]
		Consider $d\ge 3m$. We only have to prove that $L_{3,d}(m^{15})$ is non-special. Since $\vdim L_{3,d}(m^{15})>m+1$, it is enough to prove that $ L_{3,d}(m^{15},1^{m+1})$ is non-special. We apply Proposition \ref{pro:limitediottograssi} to degenerate $ L_{3,d}(m^{15},1^{m+1})$ to $ L_{3,d}(2m+1,m^7)$, and the latter is non-special by \cite[Theorem 5.3]{DL}.
	\end{proof}


Proposition \ref{pro:limitediottograssi} allows us to confirm Laface-Ugaglia conjecture for another family of linear systems.

\begin{pro}
	Let $m_1\ge\ldots\ge m_8$ be non-negative integers. Assume that
	\begin{enumerate}
		\item $\vdim L_{3,d}(m_1^8,\dots,m_8^8)>8+m_1+\ldots+m_8$;
		\item $d\ge 2(m_1+m_2)+1$.
	\end{enumerate}
	Then $ L_{3,d}(m_1^8,\dots,m_8^8)$ is non-special.
	\begin{proof}
		By assumption (1), it is enough to prove that $ L_{3,d}(m_1^8,\dots,m_8^8,1^{8+m_1+\ldots+m_8})$ is non-special. We apply Proposition \ref{pro:limitediottograssi} to degenerate $ L_{3,d}(m_1^8,\dots,m_8^8,1^{8+m_1+\ldots+m_8})$ to $ L_{3,d}(2m_1+1,\dots,2m_8+1)$. The latter is non-special by assumption (2) and \cite[Theorem 5.3]{DL}.
	\end{proof}
\end{pro}

We now aim to provide further examples of non-special linear systems.

\begin{pro}\label{pro:interpolazionesuperfici}
	Let $l,m\le 42$. Set
	\[h:=\frac{\binom{m+1}{2}}{\binom{l+1}{2}}.\]
	Assume that $h\in\N$ and $h\ge 10$.
	\begin{enumerate}
		\item Let $d\in\N$. Then $ L_{2,d}(l^h)$ is non-special and $ L_{2,d}(l^t)$ is non-special whenever $\vdim L_{2,d}(l^t)\ge (h-t)\binom{l+1}{2}$.
		\item Let $a,b\in\N$ such that $a,b\ge m$. Then $ L_{\p^1\times\p^1,(a,b)}(l^h)$ is non-special and $ L_{\p^1\times\p^1,(a,b)}(l^t)$ is non-special whenever $\vdim L_{\p^1\times\p^1,(a,b)}(l^t)\ge (h-t)\binom{l+1}{2}$.
	\end{enumerate}
	\begin{proof} We apply Proposition \ref{pro:n-pli_su_m-plo} to specialize our linear system.
		\begin{enumerate}
			\item We degenerate $ L_{2,d}(l^h)$ to $ L_{2,d}(m)$, which is always non-special. Moreover $ L_{2,d}(l^h))$ is not empty, hence $ L_{2,d}( l^t)) $ is non-special as well.
			\item We degenerate $ L_{\p^1\times\p^1,(a,b)}(l^h))$ to $ L_{\p^1\times\p^1,(a,b)}(m)$. By \cite[Theorem 1.5]{CGG}, the latter is isomorphic to $ L_{2,a+b}(a,b,m)$. Since $a+m, b+m\le a+b$, the base locus does not contain double lines, so the system is non-special by \cite[Theorem 5.1]{Cil}. Moreover $ L_{\p^1\times\p^1,(a,b)}(l^h))$ is not empty, hence $ L_{\p^1\times\p^1,(a,b)}( l^t)) $ is non-special as well.\qedhere
		\end{enumerate}
	\end{proof}
\end{pro}

\begin{pro}
Let $d,l,m\in\N$ such that $l\le 5$ and $m\ge 2l+1$. Set
	\[h:=\frac{\binom{m+2}{3}}{\binom{l+2}{3}}\]
	and assume that $h\in\N$.
	Then $ L_{3,d}(l^h)$ is non-special and $ L_{3,d}(l^t)$ is non-special whenever $\vdim L_{3,d}(l^t)\ge (h-t)\binom{l+2}{3}$. Furthermore, take $m_1,\dots,m_t\in\N$. Then $ L_{3,d}(m_1,\dots,m_7,l^h)$ is non-special under the assumption that $d>i+j$ for every $i,j\in\{m_1,\dots,m_7,m\}$.
	\begin{proof}
We apply Proposition \ref{omogenei_su_multiplo_P^3} to degenerate the system. For the first two statements we argue as in Proposition \ref{pro:interpolazionesuperfici}. For the last part, we degenerate $ L_{3,d}(m_1,\dots,m_7,l^h)$ to $ L_{3,d}(m_1,\dots,m_7,m)$, which is non-special by \cite[Theorem 5.3]{DL}.
	\end{proof}
\end{pro}

	
Let us point out that there are cases in which we can compute the limit with weaker assumptions, and therefore we can still apply this degeneration. For instance, on surfaces the hypothesis $h\ge 10$ can be relaxed.

\begin{es}
Pick $l=5$, $m=14$, and $h=7$. By a sequence of Cremona transformation, it is easy to check that $ L_{2,13}( 5^7) $ is empty and therefore non-special, so the limit of 7 collapsing $5$-tuple points in $\A^2$ is a $14$-tuple point by Lemma \ref{lem:grassisugrasso}. 
%
\end{es}

With some extra effort, we can employ a sequence of collisions to show that other linear systems are non-special.

\begin{pro}
	Let $m,n_1,\dots,n_s\in\N$. For $i\in\{1,\dots,s\}$, set $h_i:=\frac{m(m+1)}{n_i(n_i+1)}$. Assume $h_i\in\N$ and $h_i\ge 10$ for every $i\in\{1,\dots,s\}$. If $m,n_1,\dots,n_s\le 42$, then $$ L_{2,d}(m^k,n_1^{t_1h_1},\dots,n_s^{t_sh_s})$$ is non-special for every $k,t_1\,\dots,t_s\in\N$.
	\begin{proof}
		By Proposition \ref{pro:n-pli_su_m-plo}, we can collapse $h_1$ of the $n_1$-tuple points into an $m$-tuple point, thereby degenerating $ L_{2,d}(m^k,n_1^{t_1h_1},\dots,n_s^{t_sh_s})$ to
		$$ L_{2,d}(m^{k+1},n_1^{(t_1-1)h_1},\dots,n_s^{t_sh_s}).$$
		By performing $t_1$ of these collisions, we obtain the system $$ L_{2,d}( m^{k+t_1},n_2^{t_2h_2},\dots,n_s^{t_sh_s}) .$$
		Then we apply Proposition \ref{pro:n-pli_su_m-plo} again to collapse $h_2$ of the $n_2$-tuple points into an $m$-tuple point. By performing $t_2$ of these collisions, we specialize the system to
		\[ L_{2,d}(m^{k+t_1+t_2},n_3^{t_3h_3},\dots,n_s^{t_sh_s}).\]
		We iterate the argument till the $s$-th step. At the end we are dealing with the specialized system $ L_{2,d}( m^{k+t_1+\ldots+ t_s}) $. The latter is non-special by \cite[Theorem 32]{Du}, and this implies $ L_{2,d}(m^k,n_1^{t_1h_1},\dots,n_s^{t_sh_s})$ is non-special.
	\end{proof}
\end{pro}

Up to now, we could benefit from known results about non-special systems on $\p^2$ and $\p^3$. In these two cases, there are very precise conjectural classifications of special systems, and such conjectures are known to hold in many cases. However, for $n\ge 4$ not even a conjectural solution of the problem is known. For this reason, our results on $\p^4$ are limited to triple points, but they still provide hints to understand an almost unexplored topic.

\begin{pro}\label{pro:interpolazioneinp4}
	If $d\ge 8$, then $ L_{4,d}(3^r)$ is non-special for every $r\le 11$. Moreover, if $d\ge 11$ then $ L_{4,d}(3^r)$ is non-special for every $r\le 66$.
	\begin{proof}
		For the first part we only have to prove that $ L_{4,d}(3^{11})$ is non-special. Since $\vdim L_{4,d}(3^{11})>36$, it is enough to prove that $ L_{4,d}(3^{11},1^{36})$ is non-special. We apply Proposition \ref{pro:collisionetripliinp4} to degenerate $ L_{4,d}(3^{11},1^{36})$ to $ L_{4,d}(6,3^5)$. By using reducible divisors, it is easy to show that $ L_{4,d}(6,3^5)$ has a 0-dimensional base locus, and therefore it is non-special by \cite[Corollary 4.8]{BDP1}.
		
		For the second part, assume that $d\ge 11$. We only have to prove the case $r=66$. Since $\vdim L_{4,d}(3^{66})>216$, it is enough to prove that $ L_{4,d}(3^{66},1^{216})$ is non-special. Again, we use Proposition \ref{pro:collisionetripliinp4} to degenerate $ L_{4,d}(3^{66},1^{216})$ to $ L_{4,d}(6^6)$. By using reducible divisors, we see that $ L_{4,d}(6^6)$ has a 0-dimensional base locus, and therefore it is non-special by \cite[Corollary 4.8]{BDP1}.
	\end{proof}
\end{pro}

Actually, something stronger holds. Proposition \ref{pro:collisionetripliinp4} can be generalized, by proving that the collision of $n+2$ triple points and a bunch of simple points in $\p^n$ give a point of multiplicity 6. Thus we can repeat the argument of Proposition \ref{pro:interpolazioneinp4} to show that $ L_{n,8}(3^{2n+3})$ is non-special. In a similar fashion, $ L_{4,11}(4^{11})$ is non-special. However, these linear system have a large virtual dimension, so we feel that the most interesting results are the ones stated in Proposition \ref{pro:interpolazioneinp4}.

\end{document}